\documentclass{amsart}

\usepackage{graphicx,mathrsfs,amssymb,mathtools,enumitem,xcolor}

\graphicspath{ {figures/} }

\newtheorem{theorem}{Theorem}[section]
\newtheorem{lemma}[theorem]{Lemma}

\theoremstyle{definition}
\newtheorem{definition}[theorem]{Definition}

\newtheorem{corollary}[theorem]{Corollary}
\newtheorem{proposition}[theorem]{Proposition}
\theoremstyle{remark}
\newtheorem{remark}[theorem]{Remark}

\numberwithin{equation}{section}

\newcommand{\R}{\mathbb{R}}
\newcommand{\Z}{\mathbb{Z}}

\newcommand{\PGL}{\mathrm{PGL}}

\newcommand{\Aut}{\mathrm{Aut}}
\newcommand{\E}{\mathbb{E}}
\newcommand{\Isom}{\mathrm{Isom}}
\newcommand{\RP}{\mathbb{R} \textrm{P}}

\begin{document}

\title{Codimension-$1$ Simplices in Divisible Convex Domains}

\author{Martin D. Bobb}
\address{Department of Mathematics,••••••••
The University of Texas at Austin,
2515 Speedway, RLM 8.100,
Austin, TX 78712}

\begin{abstract}
Properly embedded simplices in a convex divisible domain $\Omega\subset\RP^d$ behave somewhat like flats in Riemannian manifolds \cite{Schroeder}, so we call them \textit{flats}. We show that the set of codimension-$1$ flats has image which is a finite collection of disjoint virtual $(d-1)$-tori in the compact quotient manifold. If this collection of virtual tori is non-empty, then the components of its complement are cusped convex projective manifolds with type $d$ cusps.
\end{abstract}

\maketitle

\let\thefootnote\relax

\section{Introduction}

A \textit{convex projective domain} (or \textit{set}) is a subset of real projective space which is convex, open, and properly contained in an affine chart. Such a domain $\Omega$ has a metric defined in terms of the projectively invariant cross-ratio (the \textit{Hilbert metric}), and hence the group $\Aut(\Omega)=\{g\in \PGL(d+1,\R) ~\vert~ g.\Omega=\Omega\}$ acts by isometries. If we take $\Omega$ to be a round ball, this is the projective model of $d$-dimensional real hyperbolic space.

When $\Omega$ is a convex set which is not equivalent to an ellipse, the resulting geometry is less uniform (we do not expect $\Aut(\Omega)$ to act transitively, or even to be non-trivial in general). We will take an interest in those domains $\Omega$ which permit cocompact actions by discrete subgroups of $\Aut(\Omega)$. Such a domain is called \textit{divisible}. For a general resource on divisible convex domains, we direct the reader to Benoist' survey \cite{Benoist_survey}.

If we think of convex projective geometry as a generalization of hyperbolic geometry, then we expect it to exhibit properties common to negatively curved spaces. While this is true to some extent, it is more accurate to say that convex projective domains are akin to non-positively curved spaces. A simplex $\Delta$ in $\Omega$ is \textit{properly embedded} when $\partial\Delta\subset\partial\Omega$. The Hilbert metric on a simplex is bi-Lipschitz to the metric on Euclidean space, and properly embedded simplices behave somewhat like flat subspaces of Riemannian spaces. There is a successful industry in Riemannian geometry focused on understanding topological and geometric aspects of non-positively curved manifolds by examining the structure of their zero-curvature subspaces (for example \cite{Schroeder_lowdimflats, Schroeder}). We apply similar concepts to understand divisible convex projective domains by considering their properly embedded codimension-$1$ simplices.

This work was originally inspired by Benoist' 2006 result \cite{Benoist}. Part of this result is that when an irreducible divisible convex projective domain in $\RP^3$ contains properly embedded triangles and its dividing group is torsion-free, then the quotient manifold is topologically a collection of cusped hyperbolic $3$-manifolds glued along cusp cross-sections, and the triangles in $\Omega$ are the covers of these cross sections. The method of proof involves careful analysis of segments in $\partial \Omega$ (and its dual) combined with a theorem of Morgan-Shalen granting information about $3$-manifold group actions on an $\R$-tree \cite{Morgan-Shalen} (the $\R$-tree which is dual to the set of triangles in $\Omega\subset\RP^3$). Benoist uses this data to show that the triangles in $\Omega$ cannot accumulate: that they are isolated from one another. This is enough to show that the quotients of these triangles must be disjoint virtual tori. These tori are exactly the embedded incompressible tori of the manifold's JSJ-decomposition \cite{Jaco-Shalen, Johannson}, and various data provided by Benoist precludes the constituents of this decomposition from supporting any of the $8$ Thurston geometries besides hyperbolic.

This is a beautiful theorem, and it is not immediately obvious what should generalize to higher dimensions. In particular, there is no obvious way to apply the technology of Morgan-Shalen (or any analogue) to the higher dimensional setting.

This paper demonstrates the following decomposition principle for divisible convex domains in dimension greater than or equal to $3$, which generalizes the decomposition provided by Benoist in dimension 3.

\theoremstyle{theorem}
\newtheorem*{main theorem}{Theorem \ref{theorem: main theorem}}
\begin{main theorem}
If $d\geq 3$ and $\Omega\subset \RP^d$ is an irreducible divisible convex domain, then the set $\mathscr{F}$ of properly embedded simplices of dimension $(d-1)$ is discrete. Let $\Gamma$ be a discrete group of projective transformations preserving and acting cocompactly on $\Omega$, and let $M=\Gamma \backslash \Omega$ (a manifold or orbifold). If $\mathscr{F}$ is non-empty, then its image $p(\mathscr{F})\subset M$  is a finite collection of disjoint virtual $(d-1)$-tori and the complement of $p(\mathscr{F})$ is a union of cusped convex projective manifolds whose cusps have diagonalizable holonomy.
\end{main theorem}

For a discussion of convex projective cusps see Section \ref{section: cusps} and \cite{Ballas-Cooper-Leitner, Cooper-Long-Tillmann}.

It may be shown from Benzecri's work in the 1960's that the elements of $\mathscr{F}$ do not interesect (although he did not directly state this) \cite{Benzecri}. The majority of the work is to show that codimension-$1$ properly embedded simplices form a discrete set (ie do not accumulate) in $\Omega$, and hence have compact image (which are necessarily virtually $(d-1)$-tori) (Section \ref{section: the hard work}).

\subsection{Outline of the proof of the main theorem}

Let $\mathscr{F}$ be the set of flats (codimension-1 properly embedded simplices) in $\Omega$, an irreducible convex divisible set with $\Gamma$ dividing $\Omega$.  Benzecri tells us at the outset that $F$ and $F' \in \mathscr{F}$ are disjoint, and have disjoint boundary as well \cite{Benoist} \cite{Benzecri}. Benoist makes the observation that if $F\in \mathscr{F}$ has that $stab(F)$ is not finite, then $F$ is isolated in $\mathscr{F}$ (Lemma \ref{lemma: stabilizer gives isolated}). Our goal is thus to find elements of $\Gamma$ stabilizing flats. We proceed as follows:

\begin{enumerate}
\item Let $\mathscr{F}_\partial \subset \mathscr{F}$ be the set of flats which are part of the boundary of some region in $\Omega\setminus \cup (F\in\mathscr{F})$ (\emph{boundary flats}). We adapt a lemma of Schroeder \cite{Schroeder} to the convex projective setting to demonstrate that \textbf{if $F$ is a boundary flat then $stab(F)>\Z^{d-1}$} (Proposition \ref{prop: boundary full rank}). 
\item We may now address the entirety of $\mathscr{F}_\partial$, and a brief homological argument combined with a geometric observation shows that \textbf{the set of all boundary flats $\mathscr{F}_\partial$ does not accumulate} (Proposition \ref{prop: boundary flats don't accumulate}). 
\item Next we show  that \textbf{$\Omega$ has no subsets foliated by flats}, which involves an application of Vey's irreducibility theorem and Proposition \ref{prop: boundary flats don't accumulate}. This shows that $\mathscr{F}_\partial = \mathscr{F}$ and that \textbf{$\mathscr{F}$ does not accumulate} (Proposition \ref{prop: isolated or foliated} and Corollary \ref{corollary: no completely foliated omega}). 
\item The rest of the main theorem, that \textbf{the image of $\mathscr{F}$ is a union of virtual tori and their complement is a collection of cusped convex projective manifolds with type $d$ cusps} follows without difficulty from the definitions (Section \ref{section: main theorem}).
\end{enumerate}

\section{Background}\label{section: background}

Let $\R P^d$ be the $d$-dimensional real projective space. That is for $v\in \R^{d+1}$ let $[v]=\{kv\in \R^{d+1} ~\vert~ k\in \R\}$ and take

$$\RP^d=\{[v] ~\vert~ v\in \R^{d+1} \setminus \{0\}\}.$$

The group $\PGL(d+1,\R)$ acts on $\RP^d$ with trivial kernel. The manifold $\RP^d$ has natural affine charts associated to elements of the dual projective space $(\RP^{d})^*$ (the projectivization of the dual vector space, $(\R^{d+1})^*$). That is, for each $[\eta]\in (\RP^d)^*$, there is a `stereographic projection' from $\RP^d\setminus [\ker(\eta)]$ to $A_\eta=\{v\in \R^{d+1}~\vert~ \eta(v)=1\}$ which takes $[v]$ to its unique representative so that $\eta(v)=1$. The name `affine chart' is due to the fact that the subgroup of $\PGL(d+1,\R)$ preserving $A_\eta$ is conjugate to the usual affine matrix group.

\begin{definition}
A \emph{properly convex projective domain} $\Omega$ is an open convex subset of $\RP^d$ whose closure is contained in some affine chart. The automorphism group of such a set is $\Aut(\Omega)=\{g\in\PGL(d+1,\R)~\vert~ g.\Omega=\Omega\}$.
\end{definition}

By \textit{convex}, we mean that every pair of points in $\Omega$ is contained in a segment which is also in $\Omega$.

These subsets of projective space are of particular interest because they may be equipped with an $\Aut(\Omega)$-invariant metric called the \emph{Hilbert metric} and denoted $d_\Omega$ (or simply $d$ when it is clear from context). To find $d_\Omega(x,y)$, consider the projective line containing $x$ and $y$. Due to convexity and properness of $\Omega$, there are two unique points $z_1$ and $z_2$ in $\partial \Omega$ and on the line containing $x$ and $y$ so that in a chart containing $\Omega$, these points appear in the order $z_1, x, y, z_2$. The Hilbert distance from $x$ to $y$ is then defined to be
$$d_\Omega(x,y)=\frac{1}{2} \log [z_1,x,y,z_2]$$
where $[\cdot,\cdot,\cdot,\cdot ]$ is the projective cross-ratio of four points. The reader is referred to Marquis' survey for more information about the Hilbert metric and general information about the geometry of convex domains \cite{Marquis}. 

The classical example of such a domain is the projective (Klein) model of hyperbolic geometry $\mathbb{H}^d$ found by taking $\Omega$ to be an ellipsoid (every such subset is projectively equivalent). In the case of hyperbolic space, the Hilbert metric is induced by a Riemannian metric, but this is not generally the case. It is true, however, that the Hilbert metric is induced by a norm on the tangent space $\Omega$. This gives $\Omega$ the structure of a \emph{Finsler} manifold (a manifold equipped with a continuously varying norm on its tangent space) \cite{Marquis}.

\subsection{Some convex projective geometry facts}

This subsection contains some necessary information about projective geometry. Most or possibly all of the information in this section was or is known to Benzecri and later to Benoist. However, for clarity and to help overcome any language barriers with respect to references, we include it. Perhaps the most important of these theorems is the following, due to Benzecri in the 1960's. Let $X_0$ be the set of $d$-dimensional pointed properly convex projective domains. This set may be topologized using the Hausdorff distance after choosing a metric on $\RP^d$ (there is no projectively invariant Riemannian metric on $\RP^n$ but the \textit{topology} induced on the set of closed subsets does not depend on this choice).

\begin{theorem}[Benzecri cocompactness \cite{Benzecri}]\label{theorem: Benzecri ccpt}
The group $\PGL(n+1,\R)$ acts properly and cocompactly on $X_0$.
\end{theorem}

Duality is a useful feature of projective geometry. Recall that $(\RP^d)^*$ (the space got by projectivizing $(\R^{n+1})^*$) may be naturally identified with the set of hyperplanes in $\RP^d$ by taking a class of dual vectors $[\eta]$ to the set $[\ker(\eta)]\subset \RP^d$. Note that this identification is well-defined and equivariant with respect to the action of $\PGL(d+1,\R)$ on $(\RP^d)^*$ (if $(\RP^d)^*$ is equivalence classes of row vectors, then $\PGL(d+1,\R)$ acts by right multiplication by inverse matrices).

An important observation is that if $\Omega$ is a properly convex domain, then the set of hyperplanes which are disjoint from $\overline{\Omega}$ forms a properly convex domain in $(\RP^d)^*$.

\begin{definition}[Dual domain]\label{definition: dual domain}
For any $\Omega$ a properly convex projective domain, let the \textit{dual domain} $\Omega^*=\{[\eta] ~\vert~ \ker(\eta)\cap \overline{\Omega}=\emptyset\}\subset(\RP^d)^*$.
\end{definition}

It is an exercise to verify the claim that $\Omega^*$ is a properly convex projective domain. First take a lift of $\Omega$ to its double cover, the projective sphere. Then, write representative functionals for a pair of hyperplanes in $\Omega^*$ which evaluate positively on $\Omega$, and observe that a convex combination of such functionals also evaluates positively on $\Omega$ as well. For properness of $\Omega^*$, note that any point in $\Omega$ will serve as a chart properly containing $\Omega^*$.

\subsection{Irreducibility and divisible convex domains}
As stated in the introduction, our interest is convex domains which admit quotients to compact manifolds (or orbifolds). 

\begin{definition}
If there exists a discrete subgroup $\Gamma < \Aut(\Omega)$ so that the quotient $\Gamma \backslash \Omega$ is compact, then $\Omega$ is \textit{divisible} and $\Gamma$ \textit{divides} $\Omega$.
\end{definition}

Discreteness of $\Gamma$ is enough to show that the quotient $\Gamma \backslash \Omega$ is a manifold (or orbifold), because $\Omega$ with the Hilbert metric is a proper metric space. If $\Gamma$ divides $\Omega$, then so does any finite index subgroup $\Gamma'<\Gamma$, by covering theory.

It is also necessary to discuss irreducibility of divisible domains. Let us call a component of the preimage of $\Omega$ in $\R^{d+1}$ under its projection to $\RP^d$ the \textit{cone} of $\Omega$ and denote it $C(\Omega)$. If $C(\Omega)$ can be written as a direct sum of convex cones $C(\Omega)=C_1\oplus C_2$ where $C_i\subset V_i \cap C(\Omega)$, with $V_i$  proper vector subspaces of $\R^{d+1}=V_1\oplus V_2$ then $\Omega$ is called \textit{reducible}. Otherwise, $\Omega$ is \textit{irreducible} (or sometimes \textit{indecomposable}) \cite{Benoist_survey}.

\subsection{Persistent notation in this paper}
For the remainder of the paper, we will discuss irreducible divisible convex sets of dimension $d\geq 3$. Thus, \textbf{the following notation will persist from here forward}, unless otherwise noted: $\Omega$ is an irreducible divisible convex set, $\Gamma$ is a subgroup of $\PGL(d+1,\R)$ dividing $\Omega$, $M$ is the quotient, and $p:\Omega\rightarrow M$ the quotient map. We make the standing assumption that $\Gamma$ is torsion-free. This does not restrict the class of divisible domains we are considering due to Selberg's lemma, and effects the resulting quotients $M$ by taking a finite cover \cite{Nica}. This implies that $M$ is a manifold rather than an orbifold.

\subsection{Faces and dual faces}
Let the reader be warned that the French convention is that a \textit{face} of a convex body is what is called a \textit{facet} in English, while the French \textit{facette} is called a \textit{face} in English. In particular, this painful discrepency exists between this paper and those of Benoist and Benzecri \cite{Benoist, Benzecri}.

The bulk of this information is addressed in dimension $3$ in Benoist \cite{Benoist}. In fact, the following three paragraphs are translated almost directly from Section 2.1 of \cite{Benoist}. 

A \emph{face} of a domain $\Omega$ is the intersection of $\partial \Omega$ with $P$ where $P$ is a hyperplane not intersecting  $\Omega$. The \emph{supporting sub-space} of a face is the smallest projective sub-space containing it. The \emph{cosupport} of a face is the intersection of all hyperplanes containing it. A face is \emph{angular} when its supporting sub-space and cosupport coincide.

Since faces are convex and contained in affine charts, they are homeomorphic to balls, and so have a well-defined dimension. So we will call a face whose dimension is $k$ a \emph{$k$-face}. A $(d-1)$-face is called a \emph{facet}. 

We will say that a hyperplane $\eta \in (\RP^d)^*$ \emph{supports} a face $K$ or a point $x\in \partial \Omega$ if it does not intersect $\Omega$ but contains $K$ or $x$, respectively. By the interior of a face $K$, or $\textrm{int}(K)$, we will mean the relative interior of $K$, and by $\partial K$ we will mean the complement in $K$ of its interior.

The same definitions apply to $\Omega^*$ the dual domain, where the roles of points and hyperplanes are reversed (e.g., a face of $\Omega^*$ is the intersection of $\partial{\Omega^*}$ with a point not intersecting $\Omega^*$, and so forth).

It should be noted that one may construct a convex body $\Lambda$ with a face $K$ so that $K$ as a convex body has a face $F$, but so that $F$ is not a face of $\Lambda$ (in $\R^2$, the convex hull of a circle and a point outside of the circle is such an example). Less technically, the statement `a face of a face is a face' is not always true when using this definition of a face. One may rectify this by taking an intrinsic definition of faces, and call what is herein referred to as a face an \textit{exposed} face, but we will not do so. What is true is that every point in $\partial \Omega$ is contained in \textit{some} face. See \cite{Rockafellar} for a discussion of faces of convex bodies.

The following is another result of Benzecri, and is critical here. Benoist is responsible for distilling and phrasing this result in modern language.

\begin{theorem}\label{theorem: no angular faces}\cite{Benzecri, Benoist}
If $\Omega$ is an irreducible convex divisible set, it has no angular faces. In particular, it has no facets at all, and each face of dimension $(d-2)$ has a unique supporting hyperplane.
\end{theorem}

This result is really a corollary of a result of Benzecri that divisible convex domains have closed orbits under the action of $\PGL(d+1,\R)$ in the space of convex domains.

If $K$ is a face of $\Omega$, then its \emph{dual face} $K^*\subset \partial \Omega^*$ is the union of supporting hyperplanes containing $K$. This is a face of $\Omega^*$, any point in the interior of $K$ providing the necessary hyperplane in $(\RP^d)^*$.

\begin{remark}\label{remark: angular dim count}
Note that a face $K$ of $\Omega$ is angular if and only if $dim(K)+dim(K^*)\geq d-1$.
\end{remark}

The following lemma is simple but necessary.

\begin{lemma}\label{lemma: little lemma}
Let $J$ be a convex subset of $\partial\Omega$, and $x\in \textrm{int}(J)$. Then any hyperplane $\eta$ supporting $x$ also supports $J$, and if $J^*$ is the set of hyperplanes supporting $J$, then $J^*$  is a face of $\Omega^*$.
\end{lemma}

\begin{proof}
Since $\Omega$ is properly convex, we may examine $J$ and $\eta$ in an affine chart containing $\Omega$ (so that $\eta$ is two-sided). As $\eta$ is a supporting hyperplane, in this chart $\Omega$ lies on one side of it. 

If $J$ is not contained in $\eta$, then $\eta$ divides $J$ into two components. However, points in $J$ are arbitrarily close to points in $\Omega$, so there are points of $\Omega$ on both sides of $\eta$, a contradiction. So $J$ is contained in $\eta$.

Lastly, observe that $J^*$ is obtained by taking the intersection of $\overline{\Omega}^*$ with $x\in \mathrm{int}(J)$, and that $x$ intersects $\partial \Omega^*$ but not $\Omega^*$. Hence, by definition, $J^*$ is a face of $\Omega^*$.
\end{proof}

Lemma \ref{lemma: little lemma} allows the following clear statement about the intersection behavior of faces in $\partial \Omega$. 

\begin{proposition}[4-chotomy]\label{prop: 4-chotomy}
Let $\Omega$ be a convex projective domain, and let $L$ and $K$ be a pair of faces in $\partial \Omega$. Then exactly one of the following holds.
\begin{enumerate}
\item $L=K$
\item $L\cap K=\emptyset$
\item $L\subset \partial K$ or $K \subset \partial L$
\item $\mathrm{int}(L)\cap K=\mathrm{int}(K)\cap L=\emptyset$ and $L \cap K$ is a face.
\end{enumerate}
\end{proposition}

\begin{proof}

First observe that Lemma \ref{lemma: little lemma} shows that if $L \subset K$ and $L$ includes an interior point of $K$ then $L=K$. So if $L\subset K$, either $L=K$ or $L\subset \partial K$.

If $L\subset K$ or $K\subset L$, then either (1) or (3) holds (exclusively). If neither face is a subset of the other, then they are either disjoint, which is (3), or there is a point $z\in K\cap L$, a point $k\in K\setminus L$ and a point $l\in L\setminus K$. If this is the case, then there is some hyperplane $\eta_K$ supporting $K$ but not $L$, and some hyperplane $\eta_L$ supporting $L$ but not $K$. Hence their interiors are disjoint by Lemma \ref{lemma: little lemma}. To show that this is case (4), we must prove that $L\cap K$ is a face.

Choose a fixed affine chart properly containing $\Omega$ and orient hyperplanes in this chart which do not intersect $\Omega$ so that their positive side contains $\Omega$. Then $\eta_L$ and $\eta_K$ divide this chart into four non-empty subsets with disjoint interiors, defined by the signs of $\eta_L$ and $\eta_K$. Call the closure of these four subsets \textit{chambers}. Suppose $\mu$ is a (non-trivial) convex combination of $\eta_K$ and $\eta_L$, so that $\mu$ is disjoint from the chamber containing $\Omega$ and its opposite chamber, except for on $\eta_K\cap\eta_L$. 

Let $J=\mu\cap\overline{\Omega}$. Certainly, $L\cap K\subset J$ by the definition of $\mu$, $L$, and $K$. Suppose that $z\in J$. If $z$ were not in $\eta_L\cap\eta_K$, then $z$ would be in one of the two chambers where $\eta_L$ and $\eta_K$ have opposite sign. The point $z$ cannot be in the interior of such a chamber though, as then interior points of $\Omega$ would be on the interior of such a chamber, and one of $\eta_L$ or $\eta_K$ would evaluate negatively on points in $\Omega$, a contradiction. Thus $J\subset \eta_K\cap\eta_L$, and $J=L\cap K$. Thus in the case where $\mathrm{int}(L)\cap K=\mathrm{int}(K)\cap L=\emptyset$ but $L\cap K\neq \emptyset$, we have that $L\cap K$ is a face. This is case (4).

Certainly the four arrangements are exclusive, so the proposition is proved.

\end{proof}

We also note here that the dual takes the set of faces of $\Omega$ to the set of faces of $\Omega^*$, and reverses inclusion.  

\begin{lemma}\label{lemma: dual functor on facets}
Let $K$ and $L$ be facets of $\Omega$, so that $K \subset \partial L$. Then  $L^*\subset \partial K^*$.
\end{lemma}

\begin{proof}
It is clear that at least $L^*\subset K^*$, and Lemma \ref{lemma: little lemma} demonstrates that $L^*$ must be in the boundary of $K^*$.
\end{proof}

\subsection{Euclidean isometries and virtual properties}

Many families of manifolds and groups have key properties which always hold on finite-index subgroups or finite-sheeted covers. Such a property of a manifold or group is said to be true \textit{virtually}.

\begin{definition}[Virtual properties]
Let $P$ be a property of a group (topological space). A group (topological space) which has a finite-index subgroup (finite-sheeted cover) with property $P$ is said to \textit{virtually} have the property $P$.
\end{definition}

Here are a few simple examples:
\begin{enumerate}
    \item Every finite group is virtually trivial.
    \item The non-trivial semi-direct product $\Z\rtimes \Z$ is virtually the direct product $\Z\times \Z$.
    \item The Klein bottle is virtually the $2$-dimensional torus.
    \item Every manifold is virtually orientable.
    \item The fundamental group of any manifold $M$ virtually acts trivially on the set of orientations of the universal cover $\tilde{M}$.
\end{enumerate}

The purpose of these examples is to illustrate that virtual properties of groups and spaces are not disjoint concepts, by way of subgroup correspondence. In particular, note the relationship between the second and third items, as well as the fourth and fifth items on the list.

The study of discrete subgroups of the isometries of Euclidean space is classical. We require a few of these facts, as flats share many algebraic and geometric properties with Euclidean space (see Section \ref{section: technical flats stuff}).

The following result is a combination of a theorem of Bieberbach and its extension to discrete subgroups of $\Isom(\E^n)$ which are not cocompact.

\begin{theorem}[Virtual classification of discrete $\Isom(\E^n)$ subgroups]\cite{Cheeger-Gromoll}\label{theorem: Bieberbach/soul}
Let $H<\Isom(\E^d)$ be discrete. Then $H$ is virtually free abelian of rank $k\leq d$. The action of $H$ on $\E^n$ is cocompact if and only if $k=d$.
\end{theorem}

\subsection{Cusped convex projective manifolds}\label{section: cusps}

In Theorem \ref{theorem: main theorem}, we show that any compact convex projective manifold decomposes into a collection of cusped convex projective manifolds (if the set of flats in the domain covering it is non-empty). Cusps in the convex projective setting are more diverse than in the hyperbolic setting. It is  not immediately clear what ought to be called a `cusp' in this context, but there are a number of good reasons to take the following definition.

\begin{definition}\cite{Cooper-Long-Tillmann}
A \emph{cusp} in convex projective geometry is a manifold $C$ with a convex projective structure which is homeomorphic to $\partial C\times \R_{\geq 0}$, with $\partial C$ compact and $C$ strictly convex, and which has virtually abelian holonomy .
\end{definition}

In $d$-dimensional hyperbolic geometry, all cusps are quotients of the hyperbolic horoball. In projective geometry, there are $(d+1)$ \emph{types} of horoball with associated matrix groups stabilizing them. A more complete discussion is presented by Ballas-Cooper-Leitner, along with a parametrization of these model spaces \cite{Ballas-Cooper-Leitner}. In particular, the model horoballs are given a type, which is an integer between $0$ and $d$, where the type $0$ horoball is that of hyperbolic geometry. 

While there is a lot to be said about all types of convex projective cusps, only type $d$ cusps appear in this paper. For this reason, we will take the following corollary of the main result from \cite{Ballas-Cooper-Leitner} as our definition.

\begin{corollary}\label{corollary: diagonal hol means type n}
If $C$ is a convex projective cusp, then the  holonomy of $C$ is virtually diagonalizable if and only if $C$ is a type $d$ cusp.
\end{corollary}

The reader is invited to compare the following definition to that of cusped hyperbolic manifolds.

\begin{definition}\cite{Cooper-Long-Tillmann}\label{definition: cusped conv proj manifold}
Let $\Omega$ be a properly convex projective domain, and let $M$ be the quotient of $\Omega$ by some discrete subgroup of $\Aut(\Omega)$. We will say that $M$ is a \emph{cusped convex projective manifold} when $M$ is the union of a compact manifold with boundary (the \emph{core} of $M$) and $\{C_i\}_{i=1}^k$ a finite collection of convex projective cusps.
\end{definition}

This definition does not allow for cusps that are not full-rank. In particular, one might complain that hyperbolic manifolds which are geometrically finite but have some cusp ends which are not full rank do not fit this definition, and so are not cusped convex projective manifolds. Nevertheless, this definition suits our immediate needs, so we will adopt it.

\section{Intrinsic structure of flats, duality, and consequences}\label{section: technical flats stuff}

\subsection{$\Delta^d$ is bi-Lipschitz to $\E^d$}\label{subsection: flats are almost Euclidean}

Let $\{e_1,\dots, e_{d+1}\}$ be a collection of $(d+1)$ points in $\RP^d$ which are in generic position. Let us define $\Delta^d$ the \emph{open simplex} in dimension $d$ to be the interior of the convex hull of $\{e_1, \dots, e_{d+1}\}$ (in an affine chart disjoint from $\{e_1, \dots, e_{d+1}\}$). Up to the action of $\PGL(d+1,\R)$, there is only one such domain.

The set $\Delta^d$ is a (reducible) properly convex projective domain. The group $\Aut(\Delta^d)$ is isomorphic to the semi-direct product $\R^d \rtimes S_{d+1}$ where $S_{d+1}$ is the permutation group on $(d+1)$ symbols. 

If $\{e_1,\dots, e_{d+1}\}$ is taken as a basis, then the $\R^d$ factor acts by diagonal matrices with determinant one, and the $S_d$ factor acts by permutation matrices. In this basis, each point in $\Delta^d$ can be written uniquely as a vector $v=(\exp{v_1},\dots,\exp{v_{d+1}})$ of positive real numbers whose product is $1$. Let $\phi$ be the map from $\Delta^d$ to $\E^{d+1}$ defined by
$$\phi(v)=(v_1,\dots,v_{d+1}).$$

Here $\E^{d+1}$ is taken as the vector space $\R^{d+1}$ equipped with the usual positive definite bilinear form. The map $\phi$ is injective, and its image is a totally geodesic hyperplane in $\E^{d+1}$ which is isomorphic to (and which we will refer to as) $\E^d$.

The map $\phi$ is not an isometry to its image, but it is a bi-Lipschitz equivalence.

\begin{lemma}\label{lemma: bi-Lipschitz simplex Euclidean}
The map $\phi:\Delta^d\rightarrow \E^d$ is a bi-Lipschitz equivalence.
\end{lemma}

Recall that there is a bijection between norms on a finite dimensional vector space $V$ and the set 

$$\mathscr{K}=\{K\subset V ~\vert~ K \textrm{ is compact and convex with non-empty interior}, K=-K\}.$$ 

A norm $||\cdot||_K$ is got from $K\in\mathscr{K}$ by taking vectors on $\partial K$ to have length one, and using scaling to define lengths of other vectors. Given a norm $||\cdot||$, the set $K_{||\cdot||}$ can be recovered by taking $K$ to be the unit ball of $||\cdot||$. Furthermore, when $V$ is real, the norm $||\cdot||$ is induced by an inner-product if and only if $K_{||\cdot||}$ is an ellipsoid.

\begin{proof}
The map $\phi$ induces an injective homomorphism of $\Aut(\Delta^d)$ into $\Isom(\E^d)$ by taking $\gamma\mapsto \phi\circ\gamma\circ(\phi^{-1})$ for $\gamma \in \Aut(\Delta^n)$ (by inspection). Let $p\in \Delta^d$ be the unique point so that $\phi(p)=0\in \E^d$. The map $\phi$ induces a map $d\phi:T\Delta^d\rightarrow T\E^d$ on the tangent spaces, and a vector space isomorphism from $T_p\Delta^d\rightarrow T_{0}\E^d$. Let $K_p$ be the unit ball of the norm on $T_p \Delta^d$, and let $B_0$ be the unit ball for the Euclidean metric on $T_0\E^d$. Observe that there is some positive real number $k$ such that $k^{-1}B_0\subset d\phi(K_p)\subset k B_0$ as both subsets are convex and compact with non-empty interior, and contain the origin. Further note that the homogeneity of the two spaces gives that the same number $k$ will have this property at every point. The linearity of the integral in the standard construction of a metric from a norm on the tangent space then gives that $\phi$ is a bi-Lipschitz map.
\end{proof}


Note that the translation length of elements of $\Aut(\Delta^n)$ behaves much like that of Euclidean isometries. For $\gamma \in \Aut(\Delta^n)$, let $|\gamma| = \inf_{x\in\Delta^n}(d_{\Delta^n}(x,\gamma.x))$. Observe that $|\gamma|=0$ if and only if $\gamma$ has a fixed point. Otherwise, we will refer to $\gamma$ as a \emph{translation}.

\subsection{Codimension-1 simplices (flats)}

The substructures we are interested in are \emph{properly embedded codimension-$1$ simplices}. For the remainder of the paper, we will use the following definition. Recall that $\Omega$ is always assumed to be divisible.

\begin{definition}
A \emph{flat} $F$ in a convex projective domain $\Omega\subset \RP^d$ is a subset of $\Omega$ which is projectively equivalent to $\Delta^{k}$ and which has the property that $\partial F \subset \partial \Omega$. Let $\mathscr{F}$ be the collection of all codimension-$1$ flats in $\Omega$, and $\mathscr{F}^*$ the collection of all codimension-$1$ flats in $\Omega^*$. 
\end{definition}

The discussion of the metric on flats in Subsection \ref{subsection: flats are almost Euclidean} clarifies this nomenclature. We borrow the terminology from the study of Riemannian manifolds of non-positive curvature.

\begin{remark}
For the rest of the paper we will only deal with codimension-$1$ flats, and so we will use the term \emph{flat}, the dimension being implicit.
\end{remark}

Observe that the set $\mathscr{F}$ is closed in the geometric topology. That is, any sequence of flats whose geometric limit contains a point in $\Omega$ converge to a flat. To see this, observe the following: suppose that a sequence of flats $F_i$ are contained in hyperplanes $\eta_i$ which are converging to $\eta$ with a point $x\in\Omega\cap\eta$. Then the flats $F_i$ are defined by their vertex sets, $\{z_1,\dots,z_{n}\}_i$ which up to reordering converge to $\{z_1, \dots, z_{n}\} \subset \eta$. All that is left to note is that $\{z_1 \dots z_{n}\}$ must be in general position since their convex hull contains a point $x\in\Omega$. 

\subsection{Duality for flats}

Let $F$ be an element of $\mathscr{F}$. Note that from Theorem \ref{theorem: no angular faces} (specifically that $(d-2)$-faces of $\Omega$ have unique supporting hyperplanes), the supporting hyperplanes of the faces of $F$ intersect in a unique point in $\RP^d \setminus \overline{\Omega}$. This assures that the following is well-defined. 

\begin{definition}\label{def: pseudo-dual}
Let $F\in\mathscr{F}$ be a flat. The \textit{psuedo-dual} of $F$ is the unique point contained in all $(d-2)$-faces of $F$, and is noted $\hat{F}$.
\end{definition}

This point is analogous to the dual point to a hyperplane in the projective model of hyperbolic space, but it is not defined for every hyperplane in $\Omega$. The intention of the term `pseudo-dual' is to avoid confusion with the already existing dual structures.

A consequence of Theorem \ref{theorem: no angular faces} is that flats have the following duality properties. In short, every $k$-face $K$ of a flat $F\in\mathscr{F}$ is also a face of $\Omega$, and has a dual face $K^*$ which is an $(d-2-k)$-simplex and a face of $\Omega^*$. These dual faces make up the boundary of a flat $F^*\in\mathscr{F}^*$.

\begin{lemma}
If $F\in\mathscr{F}$ and $K$ is a face of $F$ (considering $F$ as a convex domain itself), then $K$ is a face of $\Omega$.
\end{lemma}

\begin{proof}
Let $K$ be a $k$-face of $F$. Let $L_1,\dots L_{d-1-k}$ be the $(d-2)$-faces of $F$ containing $K$ and $L_1^*,\dots,L_{d-1-k}^*$ their unique supporting hyperplanes (uniqueness by Theorem \ref{theorem: no angular faces}). Then $K=\cap_{i=1}^{d-1-k}L_i$. Let $K^*$ be the set of supporting hyperplanes containing $K$, so that the convex hull of $L_1^*,\dots L_{d-1-k}^*$ is a simplex contained in $K^*$ of dimension $(d-2-k)$. Note that $K^*$ is a face of $\Omega^*$ by Lemma \ref{lemma: little lemma}.

The dimension of $K^*$ is at least $(d-2-k)$ and the dimension of $K$ is $k$. Hence $K^*$ is a face of dimension $(d-2-k)$ and the dimension of $(K^*)^*$ is $k$ because $K\subset(K^*)^*$ (Remark \ref{remark: angular dim count}). So $(K^*)^*$ is a face of $\Omega$ containing $K$ and lying in the same $k$-plane as $K$. Thus every hyperplane supporting $K$ also supports $(K^*)^*$, and $(K^*)^*$ is a subset of $L_i$ for $i=1,\dots,n-1-k$. Therefore $(K^*)^*=K$, and the Lemma is proved.
\end{proof}

\begin{lemma}\label{lemma: flats dual to flats}
Let $F\in \mathscr{F}$ and $K$ be a $k$-face of $F$. The dual face $K^*$ is exactly the convex hull of the unique supporting hyperplanes to the $(d-2)$-faces of $F$ containing $K$. Hence, there is a flat $F^*$ of $\Omega^*$ which is the dual of $F$ (in the sense that it is a flat and each of its faces is dual to a face of $F$).
\end{lemma}

\begin{proof}
It is immediate that $K^*$ at least contains the convex hull of the supporting hyperplanes to the $(d-2)$-faces meeting $K$. Additionally, any supporting hyperplane $\phi$ to $K$ must lie in the same $(d-2-k)$-plane as this convex hull, lest $K$ and $K^*$ contradict Theorem \ref{theorem: no angular faces}. Thus $\phi$ is a linear combination of hyperplanes supporting $(d-2)$-faces of $F$ and contains $\hat{F}$. 

Let $\eta_F$ be the hyperplane containing $F$ ($\eta_F$ is \textit{not} a supporting hyperplane). We have found that any supporting hyperplane $\phi$ to $K$ is incident to $\hat{F}$, and so is equal to the span of $\eta_F\cap \phi$ and $\hat{F}$. The result follows now from the fact that the simplex in every dimension is self-dual.
\end{proof}

We also obtain as a corollary the following, which is analogous to Proposition 3.1 in \cite{Benoist}. The case in dimension $3$ is straight-forward, as one only needs to consider vertices and edges.

\begin{corollary}\label{cor: flats are disj}
If $F_1$ and $F_2$ are distinct flats, then $\overline{F}_1 \cap \overline{F}_2=\emptyset$.
\end{corollary}

\begin{proof}
Suppose $\overline{F}_1$ and $\overline{F}_2$ are not disjoint. Since $\overline{F}_1$ and $\overline{F}_2$ are codimension-$1$ in $\Omega$ and properly embedded, their boundaries intersect. So there is some face of $F_1$ intersecting some face of $F_2$, say $K_1$ and $K_2$ respectively, and let these faces be the minimal faces under inclusion with this property. By Proposition \ref{prop: 4-chotomy} and the required minimality $K_1=K_2$. However if $K_1$ and $K_2$ were not vertices, they would not be minimal. So $K_1=K_2=v$ a vertex of both $F_1$ and $F_2$. By Lemma \ref{lemma: flats dual to flats}, the dual $v^*$ of this vertex is a $(d-2)$-face of $\Omega^*$, and the $(d-1)$ vertices of $v^*$ are a collection of supporting hyperplanes to $\Omega$, each supporting $(d-2)$-face of both $F_1$ and $F_2$. 

We have that $F_1$ and $F_2$ must share $(d-2)$ of their $(d-2)$-faces. However, these $(d-2)$-faces contain the entire vertex sets of both $F_1$ and $F_2$, which must thus be equal as flats are equal to the convex hulls of their vertex sets. So $F_1=F_2$ when $\overline{F}_1 \cap \overline{F}_2\neq\emptyset$.
\end{proof}

The next lemma is not deep, but is useful. It says that no face which is not contained in $\partial F$ may intersect $\partial F$.

\begin{lemma}\label{lemma: boundary segments and flats}
If $F$ is a flat, and a segment $\sigma$ in $\partial \Omega$ meets $\partial F$, then $\sigma \subset \partial F$.
\end{lemma}

\begin{proof}
Let $z$ be a point in $\partial F\cap \sigma$, and let $L$ be the minimal (under inclusion) face of $\Omega$ containing $\sigma$. The point $z$ is in some face $K$ of $F$, and we assume $K$ to be the minimal such face under inclusion. By Proposition \ref{prop: 4-chotomy} and the assumed minimality, either $K=L$ (in which case $L\subset \partial F$ as required) or $K\subset \partial L$. In the latter case, we have that $L^*\subset K^*$. However, then $L^*$ is a face of $F^*$, and $L$ is a face of $F$ by Lemma \ref{lemma: flats dual to flats}.
\end{proof}

The following is a technical lemma about Hilbert geometry necessary for Lemma \ref{lemma: bounded distance flats don't exist}.

\begin{lemma}\label{lemma: bounded distance sequences}
Suppose that $\{a_i\}_{i=1}^\infty$ and $\{b_i\}_{i=1}^\infty$ are sequences in $\Omega$ which converge to $a_\infty$ and $b_\infty$ respectively, both elements of $\partial \Omega$. If $d(a_i, b_i)$ is uniformly bounded above, then either $a_\infty=b_\infty$ or the segment $[a_\infty, b_\infty]$ is a strict subset of a segment in $\partial \Omega$.
\end{lemma}

\begin{proof}
Suppose that $a_\infty \neq b_\infty$. We may assume that $a_i$ and $b_i$ are never equal. For each $i$, let $z_i, w_i$, be the pair of points in $\partial \Omega$ which lie on the line $span(a_i, b_i)$, and let them be chosen so that they appear in the order $z_i,a_i,b_i,w_i$. The sequences $\{z_i\}_{i=1}^\infty$ and $\{w_i\}_{i=1}^\infty$ must converge, let us call the limits $z_\infty$ and $w_\infty$ respectively.

For each $i$ consider $C_i=[z_i,a_i,b_i,w_i]$, the cross ratio. By definition $C_i$ is between $1$ and $\infty$, and the assumption of the Lemma is that there is an $N$ so that $1\leq C_i \leq N$.  Hence $a_\infty\neq z_\infty$, lest $b_\infty=z_\infty$ as well. The symmetrical statement also holds, so $z_\infty, a_\infty, b_\infty$ and $w_\infty$ are four distinct colinear points in $\partial \Omega$. Thus the Lemma holds.

\end{proof}

In two of the following proofs, we encounter a pair of flats that are at bounded distance from each other `at infinity'. This is not an arrangement that can occur in a convex divisible set, as is made clear by the following lemma. This is our first result about the geometry of flats in $\Omega$ coming from the geometry of faces in $\partial \Omega$.

\begin{lemma}\label{lemma: bounded distance flats don't exist}
If $F_1$ and $F_2$ are elements of $\mathscr{F}$, and $(x_i)_{i=1}^\infty$ and $(y_i)_{i=1}^\infty$ are sequences in $F_1$ and $F_2$ converging to $x_\infty\in\partial F_1$ and $y_\infty\in \partial F_2$ respectively, then $d(x_i,y_i)$ tends to infinity.
\end{lemma}

\begin{proof}
Proposition \ref{cor: flats are disj} (flats do not intersect in $\partial \Omega$) implies that $x_\infty\neq y_\infty$, and Lemma \ref{lemma: boundary segments and flats} shows there is no segment in $\partial \Omega$ containing $x_\infty$ and $y_\infty$ either. Thus, by Lemma \ref{lemma: bounded distance sequences}, the distance between $x_i$ and $y_i$ must tend to infinity.
\end{proof}

\subsection{An analogue of normal projections}

The pseudo-dual $\hat{F}$ furnishes flats with a distinguished normal direction. This is something special in Hilbert geometry, as a norm on the tangent space (unlike an inner product) does not distinguish a normal direction.

\begin{definition}
For each point $x\in F$ let $n_F(x)$ be the \textit{normal to $F$ at $x$} (or simply \textit{the normal direction} when clear from context), defined as the line spanned by $\hat{F}$ and $x$.
\end{definition}

This is a reasonable substitute for the normal line in Riemannian geometry. In particular, it allows for the following analogue of Sublemma 2 from Schroeder \cite{Schroeder}. Schroeder is interested in producing closed flats in finite-volume quotients of Hadamard manifolds, and this sublemma ensures that if two disjoint codimension-$1$ flats come sufficiently close but are disjoint, then the normal line to one must intersect the other. Lemma \ref{lemma: close flats project} is the convex projective analogue. This will be a useful tool in Section \ref{section: main theorem}, as we must deal with the possibility that sequences of flats may accumulate, and the normal direction gives a foot-hold in this scenario. 

\begin{lemma}\label{lemma: close flats project}
There exists a number $\epsilon_\Omega>0$ so that if $F,F'\in \mathscr{F}$, $x\in F$ and $d(x,F')<\epsilon_\Omega$, then $n_F(x)\cap F' \neq \emptyset$.
\end{lemma}

\begin{figure}[h]
\centering
\includegraphics[width=.5 \textwidth]{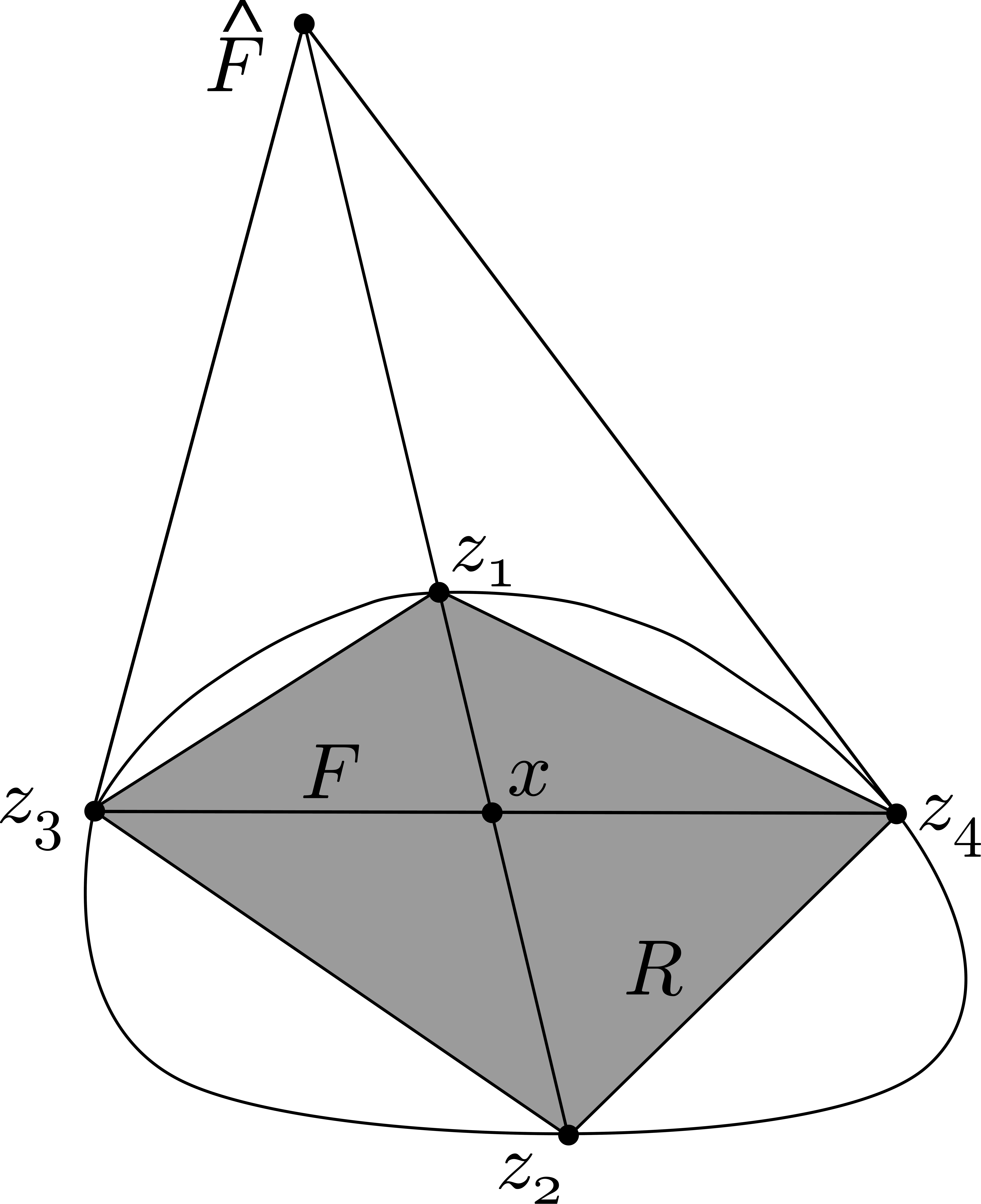}
\caption{A $2$-dimensional slice of $\Omega$ containing $x$ and $n_F(x)$.}\label{figure: 2-slice}
\end{figure}

\begin{proof}
For a fixed $x\in F$, consider any $2$-dimensional subspace $\eta$ intersecting $\Omega$ and containing $x$ and $n_F(x)$. Refer to  Figure \ref{figure: 2-slice}. The region $R$ is the convex hull in $\eta$ of the four points $z_1,\dots,z_4$ which are the intersection of $F$ with $\eta\cap\partial\Omega$ and $n_F(x)$ with $\eta\cap \partial\Omega$. Note that any line in $\eta$ containing a point in $R$ must intersect either $F$ or $n_F(x)$, and that $R$ contains some metric ball of $x$ (with respect to the Hilbert metric on $\eta\cap \Omega$). This shows that there is some lower bound $\epsilon_{\Omega,x,\eta}$ so that if $F'\in\mathscr{F}$ is within $\epsilon_{\Omega,x,\eta}$ of $x$ and intersects $\eta$, then $F'$ intersects either $F$ or $n_F(x)$.

Now we may take an infimum over the set of such slices to find a positive lower bound $\epsilon_{\Omega, x}$ for a fixed $x$ and $F$ (the set of two-planes through a point is compact). Then, using cocompactness of $\Gamma$ acting on $\Omega$, we may take an infimum over all $x\in F$ in some compact fundamental domain for this action. This infimum $\epsilon_\Omega$ is greater than zero, because $\mathscr{F}$ is closed, and hence its intersection with this fundamental domain is compact. The lemma is proved.
\end{proof}

A more brief statement of the proof of Lemma \ref{lemma: close flats project} is that $\Gamma$ acts cocompactly on the set of two-planes through points in flats, and hence we need only prove the lemma for arbitrary $2$-dimensional slices and take an infimum over this set.

\begin{remark}
Note that one could apply Benzecri's cocompactness, Theorem \ref{theorem: Benzecri ccpt} to take an infimum over all $\Omega$ and find an $\epsilon$ satisfying the conditions of the lemma and depending only on dimension.
\end{remark}

The normal line also provides a projection of $\Omega$ onto $F$, which is a distance non-increasing function. This fact is the content of the next lemma.

\begin{lemma}\label{lemma: normal projection d decreasing}
The map taking a point $y\in\Omega$ to the unique point $\pi_F(y)\in F$ which lies on the line $span(y,\hat{F})$ is a distance non-increasing projection.
\end{lemma}

\begin{proof}
Let $x, y$ be points in $\Omega$. Consider the two dimensional slice of $\RP^d$ containing $\hat{F}$, $x$, and $y$ (if these three points are not in general position, the distance under projection of $x$ and $y$ is zero, and the lemma holds). Consider Figure \ref{figure: normal distance decreasing}. The line containing $x$ and $y$ contains two points in $\partial \Omega$, let them be $z_1$ and $z_2$. The line containing $x$ and $y$ also intersects the cone from $\partial F$ to $\hat{F}$ in a pair of points, $z_1'$ and $z_2'$ which are necessarily arranged as in the figure.

Now we claim that
$$d_\Omega(x,y)=\frac{1}{2}\log [z_1,x,y,z_2] \geq \frac{1}{2}\log [z_1',x,y,z_2']=d_\Omega(\pi_F(x),\pi_F(y)).$$

\begin{figure}[h]
\centering
\includegraphics[width=.5 \textwidth]{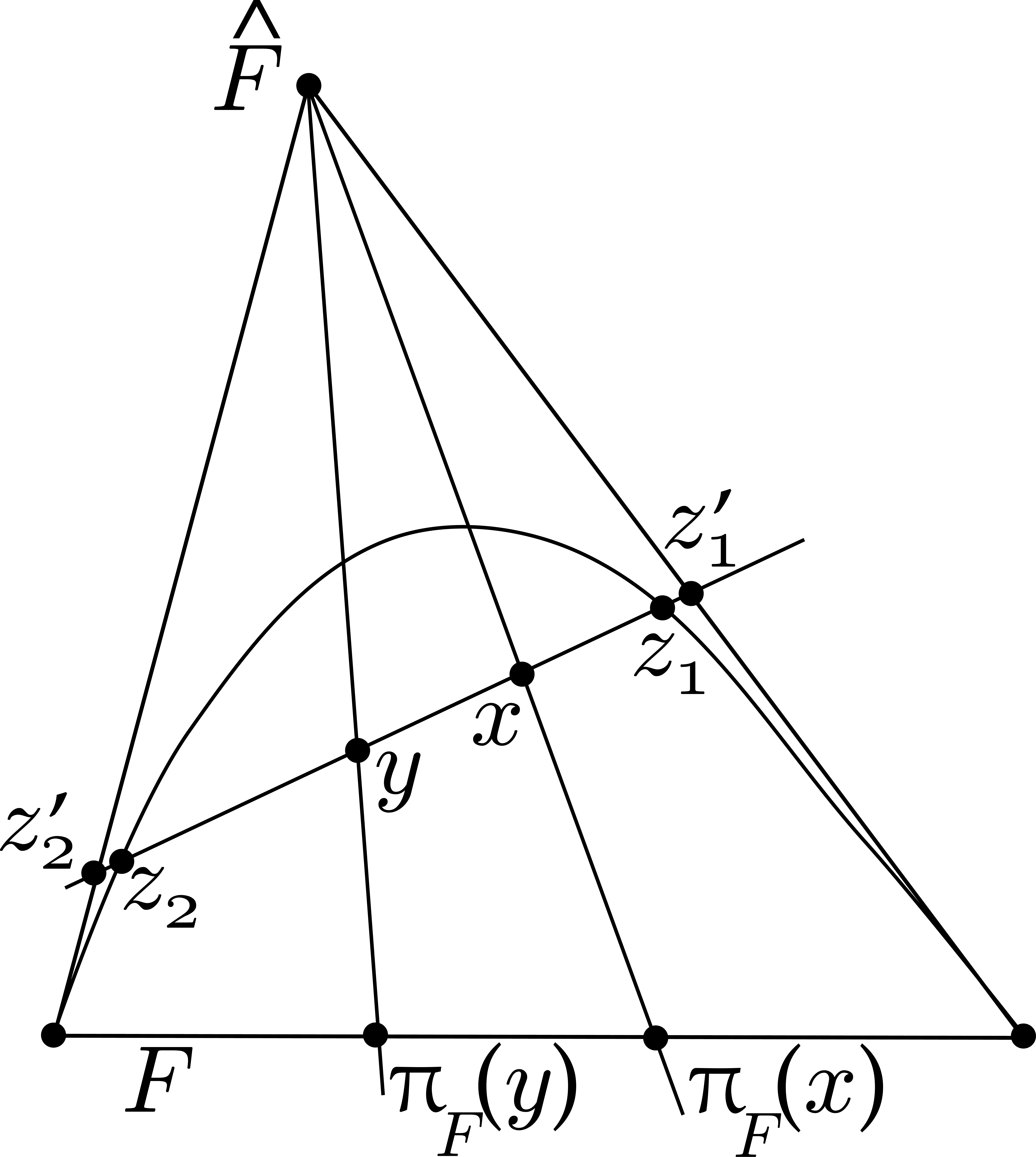}
\caption{The projection of two points onto a flat.}\label{figure: normal distance decreasing}
\end{figure}

The first equality is by definition of $d_\Omega$, the inequality is a property of the cross-ratio, combined with the fact that $\frac{1}{2}\log(\cdot)$ is an increasing function. The last equality is due to the projective invariance of the cross-ratio.
\end{proof}

We will use these properties of $n_F(x)$ and $\pi_F(x)$ in the following proofs.

\section{Discreteness of $\mathscr{F}$}\label{section: the hard work}

Having established the tools necessary to explore the space $\mathscr{F}$ we are prepared to demonstrate its discreteness in this section. Let us begin by separating $\mathscr{F}$ into two subsets, one more manageable than the other. Write
$$\mathscr{F}_\partial = \{F\in \mathscr{F}~\vert~ F\subset \partial W, \textrm{where } W \textrm{ is a component of } \Omega \setminus \cup(F'\in \mathscr{F})\}$$

and refer to the flats in this subset as \textit{boundary flats}. Let its complement be $F_{n\partial}=\mathscr{F}\setminus \mathscr{F}_\partial$, so that
$$\mathscr{F} = \mathscr{F}_\partial \sqcup \mathscr{F}_{n \partial}.$$

We will gain traction in understanding $\mathscr{F}$ by beginning with those flats bounding some open region in the complement of the set of flats, that is, we begin with $\mathscr{F}_\partial$. We will eventually show that $\mathscr{F}_{n\partial}$ is empty, but \textit{a priori} it could be that a sequence of flats accumulate like the leaves in a lamination (this is how Benoist handles this case) \cite{Benoist}. 

\subsection{Boundary flats}

The following proposition is the crucial technical result. The inspiration for the proof is Lemma 4 of Schroeder \cite{Schroeder}.

\begin{proposition}[Boundary flats have full-rank stabilizer]\label{prop: boundary full rank}
For $F$ a boundary flat, $stab(F)=\Gamma_F<\Gamma$ is virtually isomorphic to $\Z^{d-1}$
\end{proposition}

Recall first the definition of a \emph{precisely invariant} subset. Suppose a group $G$ acts on a space $Y$, and that $X\subset Y$. Then $X$ is precisely invariant when for each $g\in G$, either $g.X=X$ or $g.X\cap X=\emptyset$. It is true (and not difficult to show) that if $X$ is precisely invariant then the natural inclusion of $stab(X)\backslash X$ into $G\backslash Y$ is an injection.

\begin{proof}

Because $\Gamma_F$ is a discrete subgroup of $\Aut(F)$, it is also a discrete subgroup of $\Isom(\E^{d-1})$ (recall Lemma \ref{lemma: bi-Lipschitz simplex Euclidean}). Thus Theorem \ref{theorem: Bieberbach/soul} states that the quotient $\Gamma_F \backslash F$ is topologically $K\times \R^{(d-1-k)}$ where $K$ is a compact manifold of dimension $k$. Furthermore, we have that $\Gamma_F$ is virtually free abelian of rank $k$. Suppose that $k<d-1$, that is, suppose $\Gamma_F\backslash F$ is not compact \cite{Cheeger-Gromoll}. Let us abuse notation and identify $K$ with $K\times \{0\}$.


Let $W$ be a component of $\Omega \setminus \cup(F'\in\mathscr{F})$ so that $F$ is in $\overline{W}$, and let $\epsilon=\epsilon_\Omega$ from Lemma \ref{lemma: close flats project}.

Take $\tilde{K}\subset F$ to be a connected cover of $K$, and let $N_r(\tilde{K})$ be the metric $r$-neighborhood in $F$ of $\tilde{K}$. We will show the following claim, which we will refer to as claim $(*)$: `there is some $R>0$ so that for all $x \in F\setminus N_R(\tilde{K})$ there exists $F_x\in \mathscr{F}\setminus \{F\}$ so that $F_x$ is in $\overline{W}$ and so that $d(x,F_x) < \epsilon/2$.' Less technically claim $(*)$ is that points in $F$ succificiently far from $\tilde{K}$ are at uniformly bounded distance from other flats in $\overline{W}$. 

To see that $(*)$ is true, consider the subset 
$$Q=\{x\in W ~\vert~ d(x,F)<d(x,F') ~\textrm{for all}~ F'\in\mathscr{F}\setminus \{F\}\}.$$
The set $Q$ is precisely invariant with respect to the action of $\Gamma$ ($\Gamma$ acts by isometries and no point can be closest to two different flats), and its stabilizer is $\Gamma_F$.

If the claim $(*)$ were not true, then there is a sequence of points $\{x_i\}_{i=1}^\infty$ in $F$ whose images $p(x_i) \in stab(F)\backslash F$ leave every compact subset, but so that each $x_i$ is at distance at least $\epsilon/2$ from every other flat bounding $W$. We may assume that for all $i > j$, $x_i$ is at least distance $\epsilon/2$ from the orbit $\Gamma_F.x_j$ in $F$. For each $i$, let $y_i$ be the point on $n_F(x_i)$ which is in the direction of $W$ and which is distance $\epsilon/4$ from $x_i$, where this distance is measured along $n_F(x_i)$. For each $i$, let $B_i$ denote the radius $\epsilon/8$ ball centered at $y_i$.

The ball $B_i$ does not overlap with $\Gamma_F.B_j$ for $j<i$, and since we assumed $(*)$ to be false, infinitely many of the $B_i$ are contained in $Q$. The fact that the balls are disjoint can be seen by considering the projection onto $F$ by the normal line, which is a distance non-increasing by Lemma \ref{lemma: normal projection d decreasing}. Figure \ref{figure: key proposition} shows  the construction in dimension $3$.

\begin{figure}[h]
\centering
\includegraphics[width=.5 \textwidth]{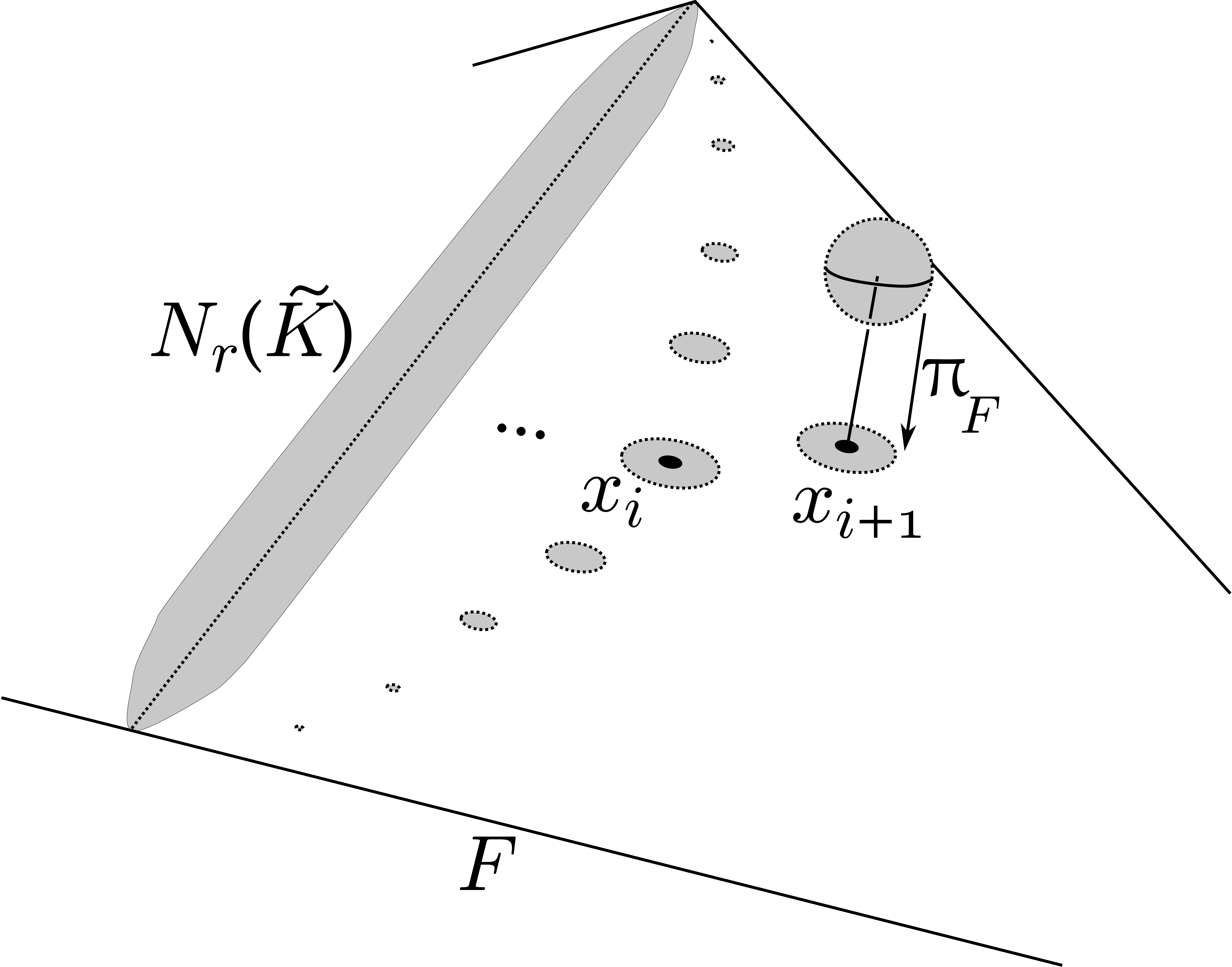}
\caption{The construction in Proposition \ref{prop: boundary full rank} when $d=3$ and $k=1$. A neighborhood of $\tilde{K}$, $\pi_F(\Gamma_F.B_i)$, $B_{i+1}$ and $\pi_F(B_{i+1})$ are shaded.}\label{figure: key proposition}
\end{figure}

Due to $Q$ being precisely invariant, $\Gamma_F \backslash Q$ embeds in $M$, which is compact. Hence, each of the $\epsilon/8$-balls $p(B_i)$ embed in $M$. However, consider the sequence $p(y_i)$ in $M$. This sequence subconverges to $q\in M$ by compactness, so there are pairs of points in this sequence which are arbitrarily close. Hence, the $\epsilon/8$ balls of these pairs must intersect, giving a contradiction with the precise invariance of $Q$. Thus claim $(*)$ holds.

Let $A=F \setminus \tilde{N}_R(K)$. Due to Lemma \ref{lemma: close flats project} and claim $(*)$, there is a well-defined map $m: A\rightarrow \mathscr{F}$ taking a point $x$ to the unique flat $m(x)\neq F$ which is in the closure of $W$ and intersects the line $n_F(x)$. This map is locally constant. This may be seen by combining Proposition \ref{cor: flats are disj} with the intermediate value theorem for continuous function. Hence, $m$ is constant on connected components of $A$.



Consider a sequence $(x_i)_{i=1}^\infty$ in a component of $A$ which converges to some $z\in\partial F \setminus \partial \tilde{K}$ and so that the image of this sequence leaves every compact set in $stab(F)\backslash F$. There is then a sequence of points $y_i\in F'=m(x_i)$ which have $d(x_i,y_i)<\epsilon_\Omega$. Hence $y_i$ converges to some $w\in \partial F'$. This directly contradicts Lemma \ref{lemma: bounded distance flats don't exist}, that bounded distance flats do not exist, and we cannot have that $k<d-1$. 

Thus, $k=dim(K)=d-1$. 
\end{proof}

The proposition then gives the following corollary immediately.

\begin{corollary}\label{cor: individual boundary flat orbits dont accumulate}
Let $F$ be a boundary flat. Then $stab(F)\backslash F$ is compact (and hence virtually a $(d-1)$-torus) and the $\Gamma$-orbit of $F$ is a discrete subset of $\mathscr{F}$.
\end{corollary}

\begin{proof}
Since $stab(F)$ is virtually $\Z^{d-1}$, we have that $stab(F)\backslash F$ is compact, and thus virtually a $(d-1)$-torus. Since $stab(F)\backslash F$ contains its limit points, $\Gamma.F$ does as well. The discreteness of $\Gamma$ as a group then gives that $\Gamma.F$ is a discrete subset of $\mathscr{F}$.
\end{proof}

The following lemma is a small extension of Proposition 3.2 (part b) in Benoist \cite{Benoist}. 

\begin{lemma}\label{lemma: stabilizer gives isolated}
A flat $F$ with infinite $\Gamma$-stabilizer is isolated in  $\mathscr{F}$.
\end{lemma}

\begin{proof}
Note that the projection $\pi_F$ defined in Lemma \ref{lemma: normal projection d decreasing} is $\Gamma_F$-equivariant. On $\partial \Omega$, this map is 2-to-1, and is bijective on each component of $\partial \Omega \setminus \partial F$. Suppose there were a sequence $F_i$ accumulating to $F$, and assume that all $F_i$ are in one component of $\Omega\setminus F$. Choose any $\gamma\in stab(F)$ which is not torsion, is orientation preserving, and is not the identity (every infinite discrete subgroup of $\Aut(\Delta^{d-1})$ has such elements, because every infinite discrete subgroup of $\Isom(\E^{d-1})$ does). The element $\gamma$ has some translation length $|\gamma|$ on $F$. Since $F_i$ accumulate to $F$, we must have that $\pi_F(\partial F_i)$ are converging to $\partial F$. So choose $F_N$ so that $\pi_F(F_N)$ contains a ball of radius at least $2|\gamma|$ in $F$. We must have then that $\pi_F(F_N)\cap \gamma.\pi_F(F_N)\neq \emptyset$, and since $\pi_F$ is $\gamma$-equivariant and bijective on components of $\partial \Omega \setminus \partial F$, we have that $\overline{F}_N$ and $\gamma.\overline{F}_N$ are not disjoint. This contradicts Corollary \ref{cor: flats are disj}.
\end{proof}

This gives the following Corollary to Proposition \ref{prop: boundary full rank}.

\begin{corollary}\label{cor: boundary flats isolated}
Boundary flats are isolated points in $\mathscr{F}$. That is, no sequence of flats in $\mathscr{F}$ accumulates to a flat in $\mathscr{F}_\partial$.
\end{corollary}

\begin{proof}
This is the application of Lemma \ref{lemma: stabilizer gives isolated} to boundary flats, which are shown to have infinite stabilizer in Proposition \ref{prop: boundary full rank}.
\end{proof}

So far, we have determined the following topological data: that the $\Gamma$-orbit of a single boundary flat is discrete, and that no boundary flat is an accumulation point in $\mathscr{F}$. The next lemma shows that the entire set of boundary flats has no accumulation points.

\begin{proposition}\label{prop: boundary flats don't accumulate}
$\mathscr{F}_\partial$ is closed and discrete as a subset of $\mathscr{F}$.
\end{proposition}

\begin{proof}

Suppose that $(F_i)_{i=1}^\infty$ was a sequence of boundary flats accumulating to $F$. The flat $F$ is necessarily in $\mathscr{F}_{n\partial}$ by Corollary \ref{cor: boundary flats isolated}. Choosing an evenly-covered neighborhood of some point in $F$, we observe that this sequence must cover an infinite collection of codimension one submanifolds in the quotient $M$. Otherwise, there would be an infinite subsequence of flats all in the same $\Gamma$-orbit, which is impossible as the $\Gamma$-orbit of a single boundary flat does not accumulate (Corollary \ref{cor: individual boundary flat orbits dont accumulate}). 

So in the quotient $M=\Gamma\backslash \Omega$ there is an infinite collection $\{p(F_i)\}_{i=1}^\infty$ of compact embedded disjoint codimension one submanifolds. 

We use a homological argument to show that the complement of this collection has infinitely many components. Alexander duality states that for $K$ compact, $H_q(M,M\setminus K)\cong H^{n-q}(K)$ (see Spanier, 6.2.1 \cite{Spanier}). Let $K$ be a union of $\alpha$ of the compact codimension-$1$ submanifolds in $M$ ($K=\cup\{p(F_i)\}_{i=1}^\alpha$), and consider the following, the end of the long exact sequence of the pair $(M,M\setminus K)$.

\begin{alignat*}{2}
        \cdots &\rightarrow H_{1}(M\setminus K)\rightarrow H_{1}(M) &&\rightarrow H_{1}(M,M\setminus K)\cong H_0(K)=\Z^\alpha\rightarrow \\
        &\rightarrow H_{0}(M\setminus K)\rightarrow H_{0}(M) &&\rightarrow H_{0}(M,M\setminus K)\rightarrow 0.
\end{alignat*} 

Apply Poincar\'{e} duality and Alexander duality to see that the final term in the first row is $\Z^\alpha$. The middle term in the first row ($H_{1}(M)$) is finitely generated, and $H_0(M)=\Z$. Thus, $H_{0}(M\setminus K)$ must have rank tending to infinity as the number of components of $K$ tends to infinity. This rank is exactly the number of connected components of $M\setminus K$.

We can assume as a result of this homological argument that $F_i$ accumulating to $F$ have images bounding different components in $M\setminus p(\cup(F\in\mathscr{F}))$, and call the regions covering these accumulating components $W_i\subset \Omega$. 
 
Let $Q_i=\{x\in W_i ~\vert~ d(x,F_i)<d(x,F'), F'\in\mathscr{F}\setminus\{F_i\}\}$ for each $i$, and observe that $Q_i$ is precisely invariant with stabilizer $stab(F_i)$. Suppose we take for each $i$ a metric ball $B_i\subset Q_i$ centered at some $y_i\in Q_i$. The points $p(y_i)\in M$ must subconverge, and by applying elements of $stab(F_i)$, we may assume that the points $y_i$ subconverge as well. Since $Q_i$ is precisely invariant, we must have that the radii of $B_i$ is approaching zero as $i$ tends to infinity. Hence the radius of the largest metric ball which may be embedded in $Q_i$ tends to zero, as the balls $B_i$ were arbitrary.

Thus, by the triangle inequality and the definition of $Q_i$, for large enough $i$, every  $x\in F_i$ is $\epsilon_\Omega$-close to some other flat, $q(x)$ (chosen to be the closest flat along $n_{F_i}(x)$ bounding $W_i$). As in the proof of Proposition \ref{prop: boundary full rank}, $q$ is locally constant, and hence has a constant value $F_i'$ on all of $F_i$. Fix such an $i$, and let $F'=q(x)$ for any $x\in F_i$ the flat which is guaranteed to be at uniformly bounded distance from $F_i$.

The existence of this pair of flats contradicts Lemma \ref{lemma: bounded distance flats don't exist}. Thus, no such accumulation may occur and the proposition is proved.
 
 
 
\end{proof}


We have now that the collection of boundary flats does not accumulate, and that boundary flats are isolated. 

\subsection{Foliated subsets of $\Omega$}\label{subsection: foliations}

Recall that we work under the standing assumption that $\Omega$ is irreducible, and thus that every point in $\Omega$ is in at most one flat.

Let us say that a closed subset $S$ of $\Omega$ is \emph{foliated by flats} if it is homeomorphic to $F\times I$, where $I$ is an interval in $\R$ containing more than one point, and so that each subset $F_a= F\times \{a\}$ is projectively equivalent to $\Delta^{d-1}$. 

One could imagine now that $\mathscr{F}$ has some discrete points, and that some subsets of $\Omega$ are foliated by flats. We focus on the flats at the boundaries of these foliated sections to demonstrate that this cannot occur.

\begin{proposition}\label{prop: isolated or foliated}
Either $\Omega$ is foliated by flats, or every flat in $\Omega$ is isolated.
\end{proposition}

\begin{proof}
Suppose that $\Omega$ has a subset $S$ which is foliated by flats. Then $S$ decomposes as $F\times I$, where $I$ is an interval. Let us assume $S$ is maximal under inclusion among foliated subsets. 

Since $\mathscr{F}$ is closed in $\Omega$, the factor $I$ must be homeomorphic to one of

\begin{enumerate}[label=(\roman*)]
    \item $[0,1]$
    \item $[0,\infty)$
    \item $(-\infty,\infty)$.
\end{enumerate}
In cases (i) and (ii), $F_0$ cannot be in $\mathscr{F}_\partial$ because it is an accumulation point of a sequence of flats in $S$ (this is the contrapositive of Corollary \ref{cor: boundary flats isolated}). Thus, $F_0$ has flats accumulated to it on both sides, or it would contradict Corollary \ref{cor: boundary flats isolated}. However, it cannot be the case that every neighborhood in $\Omega$ of $F_0$ contains points in $\Omega \setminus \cup (F\in\mathscr{F})$, for otherwise the flats bounding these regions would be accumulating to $F_0$, contradicting Proposition \ref{prop: boundary flats don't accumulate}. Thus a neighborhood of $F_0$ is foliated by flats, and $S$ was not maximal. 

So every maximal foliated subset of $\Omega$ is $F\times (-\infty,\infty)$. There may be only one such maximal subset, $S=\Omega$. If there are two such regions, $S$ and $S'$ choose a point in each and consider the segment $\sigma$ between them. The sets $S\cap \sigma$ and $S'\cap \sigma$ must be closed, connected, disjoint, and without boundary points on the interior of $\sigma$. No such pair of subsets of an interval exists, so $S$ and $S'$ cannot both exist. Similar reasoning shows that there may be no point in $\Omega$ outside of $S$.
\end{proof}

There is a theorem of Vey from 1970 which rules out the case where $\Omega$ is completely foliated by flats \cite{Vey}. Recall that a linear group $\Gamma$ is said to act \textit{strongly irreducibly} if every finite index subgroup $\Gamma'<\Gamma$ acts irreducibly.

\begin{theorem}[Vey irreducibility]\cite{Vey}\label{theorem: Vey irreducibility}
If $\Omega$ is irreducible, then $\Gamma$ acts strongly irreducibly on $\R^{d+1}$.
\end{theorem}

This allows us to conclude the following.

\begin{corollary}\label{corollary: no completely foliated omega}
No irreducible convex divisible set in dimension greater than $2$ is completely foliated by flats.
\end{corollary}

\begin{proof}
If $\Omega = F\times (-\infty,\infty)$, then let $\partial_\pm S=\lim_{t\rightarrow\pm\infty} (\overline{F_t})$ respectively, where this limit is taken in the geometric sense of closed subsets in $\RP^d$. Note that $\partial_\pm S$ are necessarily disjoint, lest flats intersect on $\partial \Omega$ (Proposition \ref{cor: flats are disj}), and that each component is contained in a hyperplane (as each $F_t$ is). Furthermore, since each flat $F_t$ is an accumulation point, no flat has non-trivial stabilizer. That is to say, the action of $\Gamma$ on the factor $(-\infty,\infty)$ in this decomposition is faithful. The kernel of this action on the orientation bundle over $(-\infty,\infty)$ then stabilizes the pair of subspaces $\partial_\pm S$ in $\partial \Omega$. This kernel is index at most two in $\Gamma$ and so $\Gamma$ does not act strongly irreducibly, contradicting Theorem \ref{theorem: Vey irreducibility}.
\end{proof}

At last we have the following:

\begin{corollary}\label{cor: no non-boundary flats and finite union of tori}
The set $\mathscr{F}_{n\partial}=\emptyset$. Equivalently, $\mathscr{F}=\mathscr{F}_\partial$. Thus, $p(\mathscr{F})$ is a finite union of virtual $(d-1)$-tori in $M$.
\end{corollary}

\begin{proof}
Suppose there were $F\in \mathscr{F}_{n\partial}$. Then $F$ is not part of a foliated region by Proposition \ref{prop: isolated or foliated} and Corollary \ref{corollary: no completely foliated omega}. Thus a point $x\in F$ has that any neighborhood of it contains a point not in a flat. Hence, $x$ is a limit of points in boundary flats, but $\mathscr{F}_\partial$ does not accumulate (Proposition \ref{prop: boundary flats don't accumulate}), and $F$ could not have existed at all.

This shows that $\mathscr{F}_{n\partial}=\emptyset$, and that $\mathscr{F}=\mathscr{F}_\partial$. To see that $p(\mathscr{F})$ is a union of virtual $(d-1)$-tori in $M$, note that we now know that every flat is a boundary flat, and so Corollary \ref{cor: boundary flats isolated} completes the proof.
\end{proof}

\section{Decomposition of convex divisible sets}\label{section: main theorem}

\subsection{Main Result}

Having now done most of what we set out to do, we may now package the conclusions of Section \ref{section: the hard work} into the main theorem. Corollary \ref{cor: no non-boundary flats and finite union of tori} states that $p(\mathscr{F})$ is a finite union of virtual tori, so let us call this collection of submanifolds $\{T_i\}_{i=1}^k$. The only challenge that remains is to show that the components of $M \setminus \{T_i\}_{i=1}^k$ are cusped convex projective manifolds. To do this, we will require special neighborhoods of these tori which have smooth and strictly convex boundary. We use the following definition.

\begin{definition}[Standard neighborhoods of flats]
Let $F\in\mathscr{F}$ be a flat. Let $D_F$ be the diagonalizable subgroup of $\PGL(d+1,\R)$ fixing the vertices of $F$ and $\hat{F}$. Let $H_F$ be the subgroup of $D_F$ which is isomorphic to $\R^{d-1}$ and contains $stab(F)$ as a lattice. When $x$ is a point in $\Omega\setminus F$ we will call the interior of the convex hull of $F$ and $H_F.x$ a \emph{standard neighborhood} of $F$.
\end{definition}

Examining the $H_F$ orbit of a point shows that standard neighborhoods have smooth boundary. Additionally, a standard neighborhood is contained in $\Omega$ for any $x\in \Omega$. Note that a standard neighborhood is not actually a topological neighborhood of $F$ as a subset of $\Omega$, as $F$ is in the boundary of a standard neighborhood.

Now we are ready for the main theorem.

\begin{theorem}[Main Result]\label{theorem: main theorem}
Let $d\geq 3$ and $\Omega\subset\RP^d$ an irreducible convex divisible domain, with $\Gamma$ dividing $\Omega$. Then the set $\mathscr{F}$ of properly embedded codimension-1 simplices (flats) in $\Omega$ is a discrete set with its geometric topology. The image of $\mathscr{F}$ in $M=\Gamma\backslash \Omega$ is a finite disjoint collection of virtual $(d-1)$-tori properly and incompressibly embedded in $M$. When $\mathscr{F}$ is not empty, the quotient of each component of $\Omega\setminus \mathscr{F}$ by its stabilizer is a cusped convex projective manifold with only type $d$ cusps.
\end{theorem}

\begin{proof}
The combination of Corollaries \ref{prop: isolated or foliated} and \ref{corollary: no completely foliated omega} shows that $\mathscr{F}$ is a discrete set. This implies that the quotient of each $F\in \mathscr{F}$ is a compact manifold as $\mathscr{F}$ is closed, and by the discussion in Subsection \ref{subsection: flats are almost Euclidean}, that each is virtually a $(d-1)$-torus. Finiteness of this collection is due to the compactness of $M$ and discreteness of $\mathscr{F}$.

Let $p: \Omega \rightarrow M$ be the quotient map to $M$, and let $\{T_i\}_{i=1}^k$ be the collection of virtual tori making up the image of $p(\mathscr{F})$. To show the last claim, that $M$ decomposes into a collection of cusped convex projective manifolds, consider some component $X$ of $M\setminus \cup (T_i\in p(\mathscr{F}))$, and take $\tilde{X}$ to be some component of $\Omega\setminus \cup (F' \in \mathscr{F})$ which covers $X$. There is some positive number $r$ so that for all $F_1,F_2\in \mathscr{F}$, the $r$-neighborhoods $N_r(F_1)\cap N_r(F_2)=\emptyset$. For each $F$, take $N(F)$ to be a standard neighborhood of $F$ which is contained in $N_r(F)$. It is clear that such a standard neighborhood exists by considering a fundamental domain $\Delta_F$ for the action of $stab(F)$ on $F$ and the cone between $\Delta_F$ and $\hat{F}$. The neighborhoods $N(F)$ are then disjoint and each is precisely invariant with stabilizer $stab(F)$.  Thus, the image $p(N(F)) \cap X$ is homeomorphic to $C\times \R_{\geq 0}$, where $C$ is a compact (Euclidean) manifold of dimension $n-1$. Furthermore, $C$ is strictly convex because $N(F)$ is. Hence, by the definition of a convex projective cusp, $p(N(F)) \cap X$ are each cusps, and they are each type $n$ by Corollary \ref{corollary: diagonal hol means type n}. The closure of $X\setminus (int(p(N(F))) \cap X)$ is compact, so by Definition \ref{definition: cusped conv proj manifold} we have that the components of $M\setminus \cup (T_i\in p(\mathscr{F}))$ are cusped convex projective manifolds with type $d$ cusps.

\end{proof}

\section{Further question}

It is important to note that it is only known that compact convex projective manifolds with properly embedded codimension-$1$ simplices exist in dimensions less than or equal to seven. These examples come from Coxeter polytopes and are described in the last section of Benoist \cite{Benoist}. 

There is, then, an obvious question about the existence of such manifolds in dimensions greater than seven. One promising method would be to produce a cusped convex projective manifold with only type $d$ cusps. The double of this manifold along its boundary would then be compact, and the cusps would become flats. See \cite{Ballas-Danciger-Lee} for an explanation of this procedure in dimension $3$.

Unfortunately, it is also not known whether or not there exist cusped convex projective manifolds with \textit{any} type $d$ cusps in dimension greater than seven, and those in dimensions four through seven come from decomposing the compact examples from Coxeter polytopes. There are some results on deforming cusped hyperbolic manifolds to change the cuspidal holonomy \cite{Bobb, Ballas-Marquis, Ballas-Danciger-Lee}. The main result of \cite{Bobb} is that there is a manifold in every dimension with a cusp which can be of any type up to $(d-1)$, but not type $d$.






\section{Acknowledgements}
The author thanks his advisor, Jeff Danciger, for helpful discourse and guidance. He also extends his thanks to Florian Stecker, Ludovic Marquis, and Harrison Bray for suggestions, and to Gye-Seon Lee for his careful reading of an early draft.

\newpage

\bibliographystyle{alpha}
\bibliography{CDDecomposition}

\end{document}